\documentclass[12pt]{article}

\usepackage[colorlinks=true,backref=page]{hyperref} 

\usepackage[labelsep=endash]{caption}

\usepackage{graphicx}
\usepackage{latexsym,amssymb}
\usepackage{amsthm}
\usepackage{indentfirst}
\usepackage{amsmath}
\usepackage{color}
\usepackage{xcolor}

\newtheorem{theoremalph}{Theorem}

\newtheorem{Theorem}{Theorem}[section]
\newtheorem*{Theorem A}{Theorem A}
\newtheorem*{Theorem A'}{Theorem A'}

\newtheorem*{Conj*}{Conjecture}

\newtheorem{Definition}[Theorem]{Definition}
\newtheorem{Proposition}[Theorem]{Proposition}
\newtheorem{Lemma}[Theorem]{Lemma}

\newtheorem*{Remark}{Remark}

\newtheorem{Remark-numbered}[Theorem]{Remark}
\newtheorem{Remarks-numbered}[Theorem]{Remarks}

\newtheorem*{Claim}{Claim}
\newtheorem{Claim-numbered}{Claim}

 \def\NN{{\mathbb N}} 

 \def\RR{{\mathbb R}}

   \def\cN{{\cal N}} 
    
   \def\cP{{\cal P}} 
\def\cE{{\cal E}}    
\def\cF{{\cal F}}

\def\triangleq{\stackrel{\triangle}{=}}

\makeatletter
\newcommand{\subsectionruninhead}{\@startsection{subsection}{2}{0mm}{-\baselineskip}{-0mm}{\bf\large}}
\newcommand{\subsubsectionruninhead}{\@startsection{subsubsection}{3}{0mm}{-\baselineskip}{-0mm}{\bf\normalsize}}
\makeatother

\begin{document}
 \title{On the growth rate of periodic orbits for vector fields}
 \author{Wanlou Wu \and Dawei Yang \and Yong Zhang\footnote{
We were partially supported by NSFC 11671288 and A Project Funded by the Priority Academic Program Development of Jiangsu Higher Education Institutions(PAPD).}}
\maketitle

\begin{abstract}
   We establish the relationship between the growth rate of periodic orbits and the topological entropy for $C^1$ generic vector fields: this extends a classical result of Katok for $C^{1+\alpha}(\alpha>0)$ surface diffeomorphisms to $C^1$ generic vector fields of any dimension. The main difficulty comes from the existence of singularities and the shear of the flow.
\end{abstract}
	
\section{Introduction}
   For an Axiom A system $f$, Bowen \cite{Bowen1} proved that the upper limit of the growth rate of periodic points is equal to its topological entropy, namely,
 $$\limsup_{n\to \infty}\frac{1}{n}\log\#P_n(f)=h_{top}(f),$$
  where $P_n(f)=\{x\in M: x=f^n(x)\}$, $h_{top}(f)$ is the topological entropy of $f$. Katok \cite[Theorem 4.3]{Katok} showed that for a $C^{1+\alpha}(\alpha>0)$ diffeomorphism $f$ on a compact manifold and any $f$-invariant Borel probability measure with non-zero Lyapunov exponents, the upper limit of the growth rate of periodic points for $f$ is larger than or equal to its metric entropy, i.e., $$\limsup_{n\to \infty}\frac{1}{n}\log\#  P_n(f)\geq h_\mu(f),~{\rm where \ }\mu {\rm \ is \ a \ hyperbolic \ measure}.$$ In particular, if $f$ is a $C^{1+\alpha}$ surface diffeomorphism,  then one has $$\limsup_{n\to \infty}\frac{1}{n}\log \# P_n(f)\geq h_{top}(f).$$
     
   In this paper, we will consider the case of vector fields. Assume that $M$ is a boundary-less compact smooth Riemannian manifold and $\mathcal {X}^1(M)$ is the space of all $C^1$ vector fields on $M$ with the $C^1$ norm. Note that ${\cal X}^1(M)$ is a Banach space. A vector field $X\in \mathcal {X}^1(M)$ generates a flow $\varphi_t=\varphi^X_t$. Suppose that $[x] \triangleq \{ y\in M: y \in {\rm Orb}(x)\}$. Let 
$$\#P_T(X)=\sum_{[x]\in P_T(X)}\pi(x),$$ 
where $\pi(x)$ is the minimal period of $x$ and $P_T(X)=\{[x]\subset M: 0<\pi(x)\leq T\}$.

\begin{theoremalph}\label{Thm:main}
   There is a dense $G_\delta$ set ${\mathcal R}\subset{\cal X}^1(M)$ such that for any $X \in {\mathcal R}$, one has $$\limsup_{T\to \infty}\frac{1}{T}\log \#P_T(X)\geq h_{top}(X):=h_{top}(\varphi_1).$$
\end{theoremalph}
    
   One of the main difficulties for proving Theorem \ref{Thm:main} is the presence of singularities. Flows with singularities have rich and complicated dynamics such as the \emph{Lorenz attractor} \cite{Lor63,Guc76}. At singularities, one can not define the \emph{linear Poincar\'e flow} (see Definition \ref{Def:linearPoincare}). Hence we lose some compact properties. Even there is no singularities, we are not able to use the usual Pesin theory as in Lian and Young \cite{Lian} since the vector field is only $C^1$. Additionally, one may have ``shear'' for flows. This is sharp in this work since we have to control the periods by the nature of this work. We introduce the outline of proof on Theorem \ref{Thm:main}:
\begin{itemize}
   \item[---] If a generic vector field is star (see Definition \ref{Def:star}), every regular ergodic measure (see Definition \ref{Def:hyperbolicmeasure}) of this vector field is hyperbolic from Shi-Gan-Wen \cite[Theorem E]{SGW14}. Then by applying a shadowing lemma, we show the relationship between the upper limit of the growth rate and the metric entropy for the regular measure.
   \item[---] If a generic vector field is not star, we prove that  the upper limit of the growth rate of periodic orbits is infinity, hence larger than the topological entropy.    
\end{itemize}

\section{Preliminaries}\label{Sec:LPF}
   Given $X\in \mathcal {X}^1(M)$, a point $\sigma\in M$ is a \emph{singularity} if $X(\sigma) = 0$. Denote by ${\rm Sing}(X)$ the set of singularities. A point $x$ is \emph{regular} if $X(x)\neq 0$. A regular point $p$ is \emph{periodic}, if $\varphi^X_{t_0}(p) = p$ for some $t_0>0$. A critical point is either a singularity or a periodic point. Denote the \emph{normal bundle} of $X$ by $$\mathcal {N}^X \triangleq \bigcup\limits_{x\in M \setminus {\rm Sing}(X)}\mathcal N_x,~~~\textrm{where}~\mathcal {N}_{x}\triangleq\{v\in T_{x}M : v\perp X(x)\}.$$

   For the flow $\varphi^X_t$ generated by $X$, its derivative with respect to the space variable is called the \emph{tangent flow} and is denoted by ${\rm d}\varphi^X_t$.
\begin{Definition}\label{Def:linearPoincare}
   Given $x\in M \setminus {\rm Sing}(X), v\in \mathcal {N}_x$ and $t\in \mathbb{R}$, the \emph{{linear Poincar\'e flow}} $\psi^X_t(v)$ is the orthogonal projection of $\varphi^X_t(v)$ on $\mathcal {N}_{\varphi^X_t(x)}$ along the flow direction $X(\varphi^X_t(x))$, $$\psi^X_t(v)\triangleq   {\rm d}\varphi^X_t(v)-\frac{\langle {\rm d}\varphi^X_t(v),X(\varphi^X_t(x)) \rangle}{\parallel X(\varphi^X_t(x))\parallel^2} X(\varphi^X_t(x)),$$ and the \emph{{rescaled linear Poincar\'e flow}} $\psi_t^*$ is $$\psi^*_t(v) \triangleq \frac{\parallel X(x)\parallel}{\parallel X(\varphi^X_t(x)) \parallel}\psi^X_t(v)=\frac{\psi^X_t(v)}{\parallel {\rm d}\varphi^X_t|_{< X(x)>}\parallel}.$$
\end{Definition}  
   Let $\mu$ be an invariant measure of $\varphi_t$ which is not concentrated on ${\rm Sing}(X)$, for the \emph{linear Poincar\'e flow} $\psi_t:\mathcal {N}\rightarrow\mathcal {N}$. by the Oseledec Theorem \cite[Theorem S.2.9]{Katok1}, for $\mu$-a.e. $x$, there is a measurable splitting $\mathcal{N}_x=\bigoplus\limits^{k(x)}_{i=1} \cE_i(x)$ and numbers $\lambda_1(x)<\lambda_2(x)<\cdots<\lambda_{k(x)}(x)$, such that $$\lim_{t\to\infty}\dfrac{1}{t}\log\parallel \psi_t(v) \parallel=\lambda_i(x),~\forall~v\in \cE_i(x)\setminus \{0\},~i=1,2,\dots,k(x),~{\rm where \ } 1\leq k(x)\leq d-1.$$ These quantities are called the \emph{Lyapunov exponents} at point $x$ of $\psi_t$ and the sub-bundle $\cE_i(x)$ is called the Oseledec subspace of $\lambda_i(x)$.
    
\begin{Definition}\label{Def:hyperbolicmeasure}
   An ergodic measure $\mu$ of the flow $\varphi_t$ is \emph{regular} if it is not supported on a singularity. A regular ergodic measure is \emph{hyperbolic}, if the Lyapunov exponents of the \emph{linear Poincar\'e flow} $\psi_t$ are all non-zero.
\end{Definition}

\begin{Remark}
   We can also define the hyperbolicity of an ergodic measure by using the tangent flow ${\rm d}\varphi_t$ as usual. However, for ergodic measures that are not supported on singularities, there will be one zero Lyapunov exponent for the tangent flow along the flow direction. 
\end{Remark}

   For a regular hyperbolic ergodic measure $\mu$, we can rewrite the splitting $\mathcal{N}=\bigoplus\limits^k_{i=1} \cE_i,~(1\leq k\leq d-1)$ as $\mathcal{N}=\mathcal{E}^s\oplus\mathcal{F}^u$, where all the Lyapunov exponents along $\mathcal{E}^s$ are negative and all the Lyapunov exponents along $\mathcal{F}^u$ are positive. We call the splitting $\mathcal{N}=\mathcal{E}^s\oplus\mathcal{F}^u$ the \emph{{hyperbolic Oseledec splitting}} w.r.t. the hyperbolic ergodic measure $\mu.$

\begin{Definition}\label{Def:dominated-poincare}
   Let $\Lambda\subset M\setminus{\rm Sing}(X)$ be an invariant (not necessarily compact) set. An invariant splitting $\mathcal{N}_\Lambda=\mathcal{E}\oplus\mathcal{F}$ w.r.t. the linear Poincar\'e flow $\psi_t$ is \emph{dominated}, if there are $C>0$ and $\eta>0$ such that for any $x\in \Lambda$ and any fixed $t> 0$, one has $$\parallel\psi_t |\mathcal{E}_x \parallel\cdot \parallel \psi_{-t}|\mathcal{F}_{\varphi_t(x)}\parallel\leq C {\rm e}^{-\eta t}.$$
\end{Definition}

   The linear Poincar\'e flow $\psi_t$ loses the compactness due to the existence of singularities. This difficulty can be overcome by extending the linear Poincar\'e flow (see \cite{LGW05}). For understanding the accumulation directions on singularity, one has to define the \emph{transgression} of the tangent flow ${\rm d}\varphi^X_t$: for the sphere bundle $SM=\{v\in TM:~\|v\|=1\}$, the map $$(t,v)\rightarrow \frac{{\rm d}\varphi^X_t(v)}{\|{\rm d}\varphi^X_t(v)\|},~~~t\in\RR,~v\in SM$$ defines a flow on $SM$. For any point $v\in SM$, one can define a fiber $\cN_v=\{u\in TM:~u\perp v\}$ and then define a bundle $\cN SM$. Then one can consider another bundle on $M$: $$NSM=\{(v_1,v_2):~v_1\in S_xM,~v_2\in T_xM,~v_1\perp v_2\},$$ and define the following flow on $NSM$ after Liao:
\begin{eqnarray*}
   \Theta_t(v_1,v_2)&=&({\rm d}\varphi^X_t(v_1),{\rm d}\varphi^X_t(v_2)-\frac{\langle {\rm d}\varphi^X_t(v_1),{\rm d}\varphi^X_t(v_2) \rangle}{\parallel {\rm d}\varphi^X_t(v_1)\parallel^2} {\rm d}\varphi^X_t(v_1))\\&=& ({\rm Proj}_S(\Theta_t),{\rm Proj}_N(\Theta_t))
\end{eqnarray*} 
   The linear Poincar\'e flow $\psi_t$ can be ``embedded'' in the flow $\Theta_t$. In fact, if one considers the subsets $\{(X(x)/\lVert X(x)\lVert,v)\}\subset NSM$, then for any regular point $x\in M$ and any $v\in \mathcal{N}_x$, one has $\psi^X_t(v)={\rm Proj}_N\Theta_t(X(x)/\|X(x)\|,v)$. 
 
   For a compact invariant set $\Lambda$, its \emph{transgression} $\widetilde\Lambda$ is defined by $${\widetilde\Lambda}={\rm Closure}\{X(x)/\|X(x)\|:~x\in\Lambda\setminus{\rm Sing}(X)\}~\textrm{in}~SM.$$ Thus, ${\rm Proj}_N\Theta_t$ is a continuous flow on $\widetilde\Lambda$. The \emph{extended linear Poincar\'e flow} $\widetilde\psi$  on ${\cal N}_{\widetilde\Lambda}SM$ is defined as the compactification of the fibered flow ${\rm Proj}_N\Theta_t$ on $\bigcup_{x\in\Lambda\setminus{\rm Sing}(X)}\{X(x)/\|X(x)\|\}$ over the base flow ${\rm Proj}_S\Theta_t$. 
         
\begin{Lemma}\label{lemma: continuous of elp} \cite[Lemma 3.1]{LGW05}
   The extended linear Poincar\'e flow $\widetilde\psi_t^X(v)$ varies continuously w.r.t. the vector field $X$, the time $t$ and the vector $v$. 
\end{Lemma} 

\begin{proof}
   This follows from the fact that the flow $\Theta_t$ is continuous w.r.t. $X$, $\RR$ and $NSM$.
\end{proof}

\begin{Lemma}\label{Lem:dominated-splitting-closure}
   If $\mathcal{N}_{\Lambda\setminus{\rm Sing}(X)}=\mathcal{E}\oplus\mathcal{F}$ is a dominated splitting w.r.t. the linear Poincar\'e flow $\psi_t$ on an invariant set $\Lambda$, then the extended linear Poincar\'e flow $\widetilde{\psi}_t$ has dominated splitting $\mathcal{N}_{\widetilde{\Lambda}}SM=\widetilde{\mathcal{E}}\oplus\widetilde{\mathcal{F}}$ and the bundles $\widetilde{\mathcal{E}},~ \widetilde{\mathcal{F}}$ are continuous on $\widetilde\Lambda$.  
\end{Lemma}

\begin{proof}
   By the definition of dominated splitting, there are $C>0$ and $\eta>0$ such that for any $x\in \Lambda\setminus{\rm Sing}(X)$ and any fixed $t>0$, one has 
\begin{eqnarray}
   \parallel\psi_t |\mathcal{E}_x \parallel\cdot \parallel \psi_{-t}|\mathcal{F}_{\varphi_t(x)}\parallel\leq C {\rm e}^{-\eta t}.\label{e.dominated-for-extended}
\end{eqnarray}
   Thus, on the set $\Gamma=\{X(x)/\|X(x)\|:~x\in\Lambda\setminus{\rm Sing}(X)\}\subset SM$, the lifts of $\cE$ and $\cF$ in $\cN_\Gamma SM$ still satisfy the inequality (\ref{e.dominated-for-extended}). Thus these two bundles can be extended on the closure of $\Gamma$, which is $\widetilde\Lambda$, in a unique and continuous way.  
\end{proof}

\begin{Lemma}\label{ms}[Definition of the transgression of a measure]
   If $\mu$ is an ergodic $\varphi_t$-invariant measure on $M$ with $\mu({\rm Sing}(X))=0$, then there is  an ergodic ${\rm Proj}_S\Theta_t$-invariant measure $\widetilde{\mu}$ on $\widetilde\Lambda$, the transgression of $\Lambda=Supp(\mu)$, such that the Lyapunov exponents of flow $\widetilde\psi_t$ w.r.t. the measure $\widetilde{\mu}$ are the same as the Lyapunov exponents of the linear Poincar\'e flow $\psi_t$ w.r.t. the measure $\mu$. The measure $\widetilde\mu$ is called the \emph{transgression} of $\mu$.  
\end{Lemma}

\begin{proof}
   Let $\mathbb{P}:TM\rightarrow M$ be the projection which is a continuous surjection with $\mathbb{P}(v)=x$ for any vector $v\in T_xM$. Take the measure $\widetilde{\mu}=(\mathbb{P}\arrowvert_{\widetilde\Lambda})_*\mu$ on $\widetilde\Lambda$. For any Borel set $A\subset\widetilde\Lambda$ with $\widetilde{\mu}(\Phi_{-t}(A))=\widetilde{\mu}(A)$ for every $t\in\mathbb{R}$, one has $\mu(\mathbb{P}({\rm Proj}_S\Theta_{-t}(A)))=\mu(\mathbb{P}A)$ for $\forall~t\in\mathbb{R}$. Since $\mu$ is an ergodic $\varphi_t$-invariant measure, $\widetilde{\mu}(A)=\mu(\mathbb{P}A)=0 {\rm \ or \ }1$. It means that $\widetilde{\mu}$ is an ergodic ${\rm Proj}_S\Theta_t$-invariant measure.
  
   Applying the Oseledec Theorem to the linear Poincar\'e flow $\psi_t:\mathcal{N}_\Lambda\rightarrow\mathcal{N}_\Lambda$, for $\mu$-a.e. $x$, there is splitting $\mathcal{N}_x =\cE_1\oplus \cE_2 \oplus\cdots\oplus \cE_m$ and quantities $\lambda_1<\lambda_2<\cdots<\lambda_m$ such that $$\lim_{t\to\infty}\dfrac{1}{t}\log\parallel \psi_t(v) \parallel=\lambda_i,~\forall~v\in \cE_i\setminus \{0\},~i=1,2,\dots,m,\text{where $1\leq m\leq$ dim$(M)-1$}.$$ By the definition of the extend Poincar\'e flow, for the splitting $\mathcal{N}_x =\bigoplus_{i=1}^m \cE_i$, $\forall v\in \cE_i\setminus \{0\}$, one has $$\lim_{t\to\infty}\dfrac{1}{t}\log\parallel \widetilde\psi_t(v) \parallel=\lim_{t\to\infty}\dfrac{1}{t}\log \parallel {\rm Proj}_N\Theta_t(X(x)/\|X(x)\|,v) \parallel=\lim_{t\to\infty}\dfrac{1}{t}\log\parallel \psi_t(v) \parallel=\lambda_i.$$ Therefore, the Lyapunov exponents of $\widetilde\psi_t$ w.r.t. the measure $\widetilde{\mu}$ are the same as the Lyapunov exponents of the linear Poincar\'e flow $\psi_t$ w.r.t. the measure $\mu$.                
\end{proof}    

\section{The Reduction of Theorem \ref{Thm:main}}  

\subsection{Proof of Theorem \ref{Thm:main}}\label{Sec:reduction}

\begin{Definition}\label{Def:star}
   A vector field $X\in \mathcal{X}^1(M)$ is \emph{star}, if there is a $C^1$ neighborhood $\mathcal{U}$ of $X$ such that every critical element of any $Y\in \mathcal{U}$ is hyperbolic.
\end{Definition}
   For a generic non-star vector field, the growth rate of periodic orbits for this vector field is infinite. We postpone the proof of Theorem \ref{Thm:non-star-infinite} in Section \ref{Sec:generic-non-star}.     
\begin{Theorem}\label{Thm:non-star-infinite}
   There is a dense $G_\delta$ set ${\mathcal R}\subset{\cal X}^1(M)$ such that if $X\in\mathcal R$ is not star, then $$\limsup_{T\rightarrow\infty}\dfrac{1}{T}\log\#P_T(X)=+\infty.$$ 
\end{Theorem}
    
   For star vector fields, we have two steps. First, based on that any regular ergodic measure of star vector field is hyperbolic, we show that the hyperbolic Oseledec splitting is a dominated splitting (Theorem \ref{Thm:dominated-measure}). Secondly, we prove that if the hyperbolic Oseledec splitting w.r.t. a regular hyperbolic measure is a dominated splitting, then the growth rate of periodic orbits is larger than or equal to the metric entropy (Theorem \ref{Thm:measure-exponential}).  

\begin{Theorem}\label{Thm:dominated-measure}
   If $\mu$ is a regular ergodic invariant measure of a $C^1$ star vector field $X$ with $h_\mu(X)>0$, then $\mu$ is a hyperbolic measure and its hyperbolic Oseledec splitting $\mathcal{N}=\mathcal{E}^s\oplus\mathcal{F}^u$ is a dominated splitting.  
\end{Theorem}
   The proof of Theorem \ref{Thm:dominated-measure} is in Section \ref {Proof}. Recall the classical definition of the entropy. For a homeomorphism $f:M\rightarrow M$, the metric $d^f_n$ defined as $d^f_n(x,y)=\max\limits_{0\leq i\leq n-1} d(f^ix,f^iy)$, the topological entropy $h_{top}(f)$ is defined as $$h_{top}(f)= \lim_{\varepsilon\rightarrow 0}\limsup_{n\rightarrow \infty}\dfrac{\log N_f(n,\varepsilon)}{n},$$where $N_f(n,\varepsilon)$ is the minimal number of $\varepsilon$-balls in the $d^f_n$ metric covering the space $M$. Katok \cite[Theorem 1.1]{Katok} defined the metric entropy $h_\mu(f)$ of $f$-invariant ergodic measure $\mu$ as $$h_\mu(f)= \lim_{\varepsilon\rightarrow 0}\lim_{n\rightarrow \infty}\dfrac{\log N_f(n,\varepsilon,\delta)}{n},$$where $N_f(n,\varepsilon,\delta)$ is the minimal number of $\varepsilon$-balls in the $d^f_n$ metric covering the set of measure larger than or equal to $1-\delta$.

\begin{Theorem}\label{Thm:measure-exponential}
   Let $\mu$ be a regular ergodic invariant hyperbolic measure of $X \in \mathcal {X}^1(M)$. If the hyperbolic Oseledec splitting $\mathcal{N}=\mathcal{E}^s\oplus\mathcal{F}^u$ is a dominated splitting, then $$\limsup_{T\to \infty}\frac{1}{T}\log \#P_T(X)\geq h_\mu(X):=h_\mu(\varphi_1).$$
\end{Theorem}
   For this theorem, we have to deal with the re-parametrization problem. In Liao's shadowing lemma, the period of periodic point which shadows the recurrent point is re-parametrization of the recurrent time. For our goal, we have to estimate the difference between the recurrent time and its re-parametrization. In section \ref{Sec:shadow-measure}, we give the shadowing lemma with time control. Theorem \ref{Thm:measure-exponential} is proved in section \ref{Sec:periodic orbits}.      

\begin{proof}[Proof of Theorem \ref{Thm:main}]
   Take a dense $G_\delta$ set $\mathcal{R}\subset{\cal X}^1(M)$ as Theorem \ref{Thm:non-star-infinite}. For any $X\in\mathcal{R}$, if $X$ is not star, by Theorem \ref{Thm:non-star-infinite}, one has $$\limsup_{T\to \infty}\frac{1}{T}\log \#P_T(X)=\infty> h_{top}(X).$$ If $X$ is star, then any ergodic invariant measure $\mu$ of $X$ is a hyperbolic measure by \cite[Theorem E]{SGW14}. We can suppose that $h_\mu(X)>0$, since if $h_\mu(X)=0$, the inequality is true. According to Theorem \ref{Thm:dominated-measure}, the hyperbolic Oseledec splitting $\mathcal{N}=\mathcal{E}^s\oplus\mathcal{F}^u$ w.r.t. the ergodic invariant measure $\mu$ is a dominated splitting. By Theorem \ref{Thm:measure-exponential}, one has $$\limsup_{T\to \infty}\frac{1}{T}\log \#P_T(X)\geq h_\mu(X).$$ By the variational principle, $h_{top}(\varphi_1) = \sup~ \{h_\mu(\varphi_1): \mu {\rm \ is \ an \ ergodic \  measure \ of \ X }\}$. Thus one has $$\limsup_{T\to \infty}\frac{1}{T}\log \#P_T(X)\geq h_{top}(X).$$ The proof of Theorem \ref{Thm:main} is complete. 
\end{proof}

\subsection{The proof of Theorem \ref{Thm:dominated-measure}}\label{Proof}

 Liao has proved the following estimates for star flows in \cite[Proposition 4.4]{Lia79}.       
\begin{Lemma}\label{Lem:liao-estimates}  
   For every star vector field $X\in\mathcal{X}^1(M)$, there are a $C^1$ neighborhood $\mathcal{U}$ of $X$ and constants $\eta>0,~T_0>0$, such that for any periodic orbit $\mathcal{O}$ of every $Y\in\mathcal{U}$ with $\pi(\mathcal{O})\geq T_0$ and the natural hyperbolic splitting $\mathcal{N}_{\mathcal{O}}=\mathcal{E}\oplus\mathcal{F}$ w.r.t. $\psi^Y_t$, we have
\begin{enumerate}
   \item For any $x\in \mathcal{O}$ and every $t\geq T_0$, one has 
   $\dfrac{\parallel\psi^Y_t|\mathcal{E}_x\parallel}{m(\psi^Y_t|\mathcal{F}_x)}\leq {\rm e}^{-2\eta t};$ 
   \item For any $x\in \mathcal{O}$, $$\prod^{[\pi(\mathcal{O})/T_0]-1}_{i=0}\parallel\psi^Y_{T_0}|\mathcal{E}_{\varphi^Y_{iT_0}(x)}\parallel\leq {\rm e}^{-\eta\pi(\mathcal{O})},~~\prod^{[\pi(\mathcal{O})/T_0]-1}_{i=0} m(\varphi^Y_{T_0}|\mathcal{F}_{\varphi^Y_{iT_0}(x)})\geq {\rm e}^{\eta\pi(\mathcal{O})}.$$
\end{enumerate}
\end{Lemma}   
   Next, we introduce the ergodic closing lemma for flows and give the statement about the relationship between periodic orbits and metric entropy.     

\begin{Definition}
   A point $x\in M\setminus{\rm Sing}(X)$ is \emph{strongly closable}, if for any $C^1$ neighborhood $\mathcal{U}$ of $X$, $L>0$ and any $\delta>0,$ there are $Y\in \mathcal{U}$, $y\in M$ and $\tau_0>0$ such that
\begin{enumerate}
   \item $\varphi^Y_{\tau_0}(y)=y;$ 
   \item $d(\varphi^Y_t(y),\varphi^X_t(x))<\delta,~~\forall~~0\leq 		t\leq \tau_0;$
   \item $X=Y$ on $M\backslash B$, \text {where} $B=\bigcup\limits _{t\in[-L,0]}B(\varphi^X_t(x),\delta)$.
\end{enumerate}
\end{Definition}
   Denote by $\varSigma(X)$ the set of strongly closable points of $X$, Wen \cite[Theorem 3.9]{Wen3} gave the following flow version of the ergodic closing lemma.
\begin{Theorem}\label{T2}(\cite[Theorem 3.9]{Wen3})
   For any $C^1$ vector field $X$ and any $\varphi^X_t$-invariant Borel probability measure $\mu$, one has $\mu({\rm Sing}(X)\cup\varSigma(X))=1$.
\end{Theorem}

\begin{proof}[Proof of Theorem \ref{Thm:dominated-measure}]
   According to \cite[Theorem E]{SGW14}, $\mu$ is a hyperbolic measure. Let $\mathcal{N}=\mathcal{E}^s\oplus\mathcal{F}^u$ be the hyperbolic Oseledec splitting w.r.t. the measure $\mu$. By Lemma \ref{Lem:liao-estimates}, there are $\eta>0,~T_0>0$ and a $C^1$ neighborhood $\mathcal{U}$ of $X$, such that for every periodic orbit $\mathcal{O}$ of any $Y\in\mathcal{U}$ with $\pi(\mathcal{O})\geq T_0$ and the natural hyperbolic splitting $\mathcal{N}_{\mathcal{O}}=\mathcal{E}\oplus\mathcal{F}$ w.r.t. $\psi^Y_t$, one has $$\dfrac{\parallel\psi^Y_t|\mathcal{E}_x\parallel}{m(\psi^Y_t|\mathcal{F}_x)}\leq {\rm e}^{-2\eta t},~~\forall~t\ge T_0,\forall~x\in\mathcal O;$$ $$\prod^{[\pi(\mathcal{O})/T_0]-1}_{i=0}\parallel\psi^Y_{T_0}|\mathcal{E}_{\varphi^Y_{iT_0}(x)}\parallel\leq {\rm e}^{-\eta\pi(\mathcal{O})},~~\prod^{[\pi(\mathcal{O})/T_0]-1}_{i=0} m(\psi^Y_{T_0}|\mathcal{F}_{\varphi^Y_{iT_0}(x)})\geq {\rm e}^{\eta\pi(\mathcal{O})},\forall~x\in\mathcal O.$$
  Since $h_\mu(X)>0$ and the metric entropy on any critical element is zero, we will assume that $\mu$ does not support on any critical element for the rest of proof. Define $$B(\mu)=\{x: \lim\limits_{T\rightarrow\infty}\dfrac{1}{T}\int^T_0f(\varphi_t(x))dt =\int fd\mu, \forall~\text{continuous function $f:M\rightarrow\mathbb{R}$}\}.$$ Since $\mu$ is ergodic, one has $\mu(B(\mu))=1$. Thus, one has $\mu(B(\mu)\cap Supp(\mu)\cap\varSigma(X))=1$ by Theorem \ref{T2}. For $\mu$-a.e. $x$, there are vector fields $\{X_n\}_{n\in\mathbb{N}^+}\subset\mathcal{U}$, points $\{x_n\}_{n\in\mathbb{N}^+}\subset M$ and $\{\tau_n:\tau_n>0\}_{n\in\mathbb{N}^+}$ with $\varphi^{X_n}_{\tau_n}(x_n)=x_n$ satisfying:
\begin{itemize}
\item $d(\varphi^{X_n}_t(x_n),\varphi^X_t(x))<\frac{1}{n},$ for $\forall~t\in[0,\tau_n];$\\
\item $\lVert X_n-X\lVert_{C^1}<\frac{1}{n}.$
\end{itemize}
Consider the ergodic measure $\mu_n$ which is supported on the orbit of $x_n$. Since $x$ is strongly closable, for any continuous function $f$, one has $$\lim_{n\rightarrow\infty}\int fd\mu_n=\lim_{n\rightarrow\infty}\dfrac{1}{\tau_n}\int^{\tau_n}_0f(\varphi_t(x_n))dt = \lim_{n\rightarrow\infty}\dfrac{1}{\tau_n}\int^{\tau_n}_0f(\varphi_t(x))dt = \int fd\mu.$$ Thus, $\mu_n\to \mu$ in the sense of weak$^*$ topology. As $\mu$ is not supported on any critical element, one has $\tau_n\rightarrow\infty$ as $n\rightarrow\infty$.  
\begin{Claim}
There are only finite sinks or sources among $\{Orb(x_n)\}$.
\end{Claim}   
\begin{proof}
If not, for fixed $x\in B(\mu)\cap Supp(\mu)\cap\varSigma(X)$, we may assume that $Orb(x_n)$ are sinks, then one only has 
$$\prod^{[\tau_n/T_0]-1}_{i=0}\parallel\psi^{X_n}_{T_0}(\varphi^{X_n}_{iT_0}(x_n))\parallel\leq {\rm e}^{-\eta\tau_n}.$$ 
By the definition of the extended linear Poincar\'e flow, one has that
$$\prod^{[\tau_n/T_0]-1}_{i=0}\parallel\widetilde{\psi}^{X_n}_{T_0}(\varphi^{X_n}_{iT_0}(x))\parallel \leq {\rm e}^{-\eta\tau_n}.$$ 
Since $d(\varphi^{X_n}_t(x_n),\varphi^X_t(x))<\frac{1}{n}$ for any $t\in[0,\tau_n]$ and $\lVert X_n-X\lVert_{C^1}<\frac{1}{n}$, one has that 
$$\prod^{[\tau_n/T_0]-1}_{i=0}\parallel\widetilde{\psi}^X_{T_0}(\varphi^X_{iT_0}(x))\parallel\leq {\rm e}^{-\eta\tau_n}.$$ 
By Lemma~\ref{ms}, the definition of the transgression of a measure, one has that $\int\log\parallel\widetilde{\psi}^X_{T_0}\parallel d\widetilde{\mu}(x)\leq -\eta$. It means that $\int\log\parallel\psi^X_{T_0}\parallel d\mu(x)\leq -\eta$. Therefore, the Lyapunov exponents of the linear Poincar\'e flow $\psi_t$ are negative. By the Ruelle inequality \cite[Theorem 2]{RD}, one can get $h_\mu(\varphi_{T_0})=0$. Since $\mu$ is an ergodic measure, $h_\mu(\varphi_{T_0})=\arrowvert T_0\arrowvert h_\mu(\varphi_1)=\arrowvert T_0\arrowvert h_\mu(X)>0$. This is a contradiction. The claim is thus proved.
\end{proof}   
   By Lemma \ref{Lem:liao-estimates}, for the non-trivial hyperbolic splitting $\mathcal{N}_{Orb(x_n)}=\mathcal{E}\oplus\mathcal{F}$ w.r.t. $\psi^{X_n}_t$,    
\begin{equation} \label{E2}
\begin{split} 
   \prod^{[\tau_n/T_0]-1}_{i=0}\parallel\psi^{X_n}_{T_0}|\mathcal{E}_{\varphi^{X_n}_{iT_0}(y)}\parallel\leq {\rm e}^{-\eta\tau_n},~~\prod^{[\tau_n/T_0]-1}_{i=0}m(\psi^{X_n}_{T_0}|\mathcal{F}_{\varphi^{X_n}_{iT_0}(y)})\geq {\rm e}^{\eta  \tau_n},~\forall~y\in Orb(x_n).
\end{split} 
\end{equation}
   We may assume that the indices of $Orb(x_n)$ are the same, then there is a dominated splitting on the limit point $x$ as $\mathcal{N}_x=\mathcal{G}_x\oplus\mathcal{H}_x$, where $\mathcal{G}_x=\lim\limits_{n\rightarrow\infty}\mathcal{E}_{x_n}$ and $\mathcal{H}_x=\lim\limits_{n\rightarrow\infty}\mathcal{F}_{x_n}$. We only need to prove that $\mathcal{G}_x=\mathcal{E}^s_x$ and $\mathcal{H}_x=\mathcal{F}^u_x$. By Lemma \ref{ms}, the inequalities (\ref{E2}) means that 
   $$\int\log\parallel\widetilde{\psi}^{X_n}_{T_0}|\mathcal{E}_x\parallel d\widetilde{\mu}_n(x)\leq -\eta,~~\int\log m(\widetilde{\psi}^{X_n}_{T_0}|\mathcal{F}_x)d\widetilde{\mu}_n(x)\geq\eta.$$ 
   By Lemma~\ref{lemma: continuous of elp}, since $\lVert X_n-X\lVert_{C^1}<\frac{1}{n}$ for every $n\in\mathbb{N}^+$, one has $$\int\log\parallel\widetilde{\psi}^X_{T_0}|\mathcal{E}_x\parallel d\widetilde{\mu}_n(x)\leq -\eta,~~\int\log m(\widetilde{\psi}^X_{T_0}|\mathcal{F}_x)d\widetilde{\mu}_n(x)\geq\eta.$$ By Lemma \ref{lemma: continuous of elp} again, one has  $\int\log\parallel\widetilde{\psi}^X_{T_0}|\mathcal{G}_x\parallel d\widetilde{\mu}(x)\leq -\eta$ and $\int\log m(\widetilde{\psi}^X_{T_0}|\mathcal{H}_x)d\widetilde{\mu}(x)\geq\eta$. According to the Birkhoff ergodic theorem and Lemma \ref{ms}, one has $$\lim_{m\rightarrow \infty}\dfrac{1}{m}\sum^{m-1}_{i=0}\log\parallel\psi^X_{T_0}|\mathcal{G}_{\varphi_{iT_0}(x)}\parallel=\lim_{m\rightarrow \infty}\dfrac{1}{m}\sum^{m-1}_{i=0}\log\parallel\widetilde{\psi}^X_{T_0}|\mathcal{G}_{\varphi_{iT_0}(x)}\parallel=\int\log\parallel\widetilde{\psi}^X_{T_0}|\mathcal{G}_x\parallel d\widetilde{\mu}(x)\leq -\eta,$$ $$\lim_{m\rightarrow \infty}\dfrac{1}{m}\sum^{m}_{i=1}\log m(\psi^X_{T_0}|\mathcal{H}_{\varphi_{iT_0}(x)})=\lim_{m\rightarrow \infty}\dfrac{1}{m}\sum^{m}_{i=1}\log m(\widetilde{\psi}^X_{T_0}|\mathcal{H}_{\varphi_{iT_0}(x)}) =\int\log m(\widetilde{\psi}^X_{T_0}|\mathcal{H}_x) d\widetilde{\mu}(x)\geq\eta.$$ 
  It means that $\lim\limits_{m\rightarrow \infty}\dfrac{1}{m}\sum\limits^{m-1}_{i=0}\log\parallel\psi^X_{T_0}|\mathcal{G}_{\varphi_{iT_0}(x)}\parallel\leq-\eta<0.$ That implies $\mathcal{G}_x\subset\mathcal{E}^s_x$. If $\mathcal{E}^s_x\nsubseteq\mathcal{G}_x$, then there is a non-zero vector $v$ belong to $\mathcal{E}^s_x$ but not belong to $\mathcal{G}_x$. One has the dominated splitting $v= v_1 +v_2~{\rm and \ }v_1\in\mathcal{G}_x,~0\neq v_2\in\mathcal{H}_x$. Therefore, $$\lim\limits_{t\rightarrow\infty}\dfrac{1}{t}\log \lVert \psi^X_t(v)\rVert=\lim\limits_{t\rightarrow+\infty}\dfrac{1}{t}\log \lVert \psi^X_t(v_2)\rVert\geq\lim_{m\rightarrow \infty}\dfrac{1}{m}\sum^{m}_{i=1}\log m(\widetilde{\psi}^X_{T_0}|\mathcal{H}_{\varphi_{iT_0}(x)})>0.$$ This contradicts to the fact that the Lyapunov exponents along $\mathcal{E}^s_x$ are negative. Consequently, one has that $\mathcal{E}^s_x\subset\mathcal{G}_x$. Therefore, one has that $\mathcal{E}^s_x=\mathcal{G}_x$. Similarly, one can get that $\mathcal{F}^u_x=\mathcal{H}_x$.
\end{proof}

\subsection{Proof of Theorem \ref{Thm:non-star-infinite}}\label{Sec:generic-non-star}
\begin{Lemma}\label{L8}
   There is a dense $G_\delta$ set $\mathcal{R}\subset{\cal X}^1(M)$ such that for given $T,~k\in \mathbb{N}^+$, if for every $C^1$ neighborhood $\mathcal{U}$ of $X\in \mathcal{R}$, there is $Y\in \mathcal{U}$ having $k$ periodic orbits whose periods belong to $(T,\frac{3T}{2})$, then $X$ has $k$ orbits whose periods belong to $(T,\frac{3T}{2})$. 
\end{Lemma}

\begin{proof}
   Fix a countable topological base $\{O_1,O_2,\dots,O_i,\dots\}$ of $M$. Let $\{U_1,U_2,\dots,U_n,\dots\}$ be the family of finite unions of $\{O_i\}$. We define $$\mathcal{H}^k_{n,T}\triangleq\{X: X\text{ has $k$ hyperbolic periodic orbits with period belonging to $(T,3T/2)$ in $U_n$} \},$$
\begin{eqnarray*}
   \mathcal{N}^k_{n,T}\triangleq\{~X: \exists~C^1 {\rm  \ neighborhood \ \mathcal{U} \ of \ X },~s.t.~\text{for any $Y\in\mathcal{U}$, either $Y$ has}\\ \text{ no $k$ periodic orbits with periods belonging to $(T,3T/2)$ or all $k$ periodic}\\ \text{ orbits with period belonging to $(T,3T/2)$ of $Y$ are not in}~U_n~\}.
\end{eqnarray*}
   By the definition, the set $\mathcal{N}^k_{n,T}$ is open. By the stability of hyperbolicity, $\mathcal{H}^k_{n,T}$ is open. 
\begin{Claim}
   $\overline{\mathcal{H}^k_{n,T}\cup \mathcal{N}^k_{n,T}}=\mathcal{X}^1(M).$
\end{Claim}
\begin{proof}[Proof of Claim:]
   For any $X\in\mathcal{X}^1(M)$, if $X\notin \mathcal{N}^k_{n,T}$, then for any $C^1$ neighborhood $\mathcal{U}$ of $X$, there is a $Y\in\mathcal{U}$ which has $k$ periodic orbits whose periods belong to $(T,3T/2)$ belonging to $U_n$. Thus, there is a sequence $\{X_m\}_{m\in\mathbb{N}^+}\subset \mathcal{H}^k_{n,T}$ such that $X_m\rightarrow X$. Therefore, $X\in\overline{\mathcal{H}^k_{n,T}}$. Consequently, $\mathcal{H}^k_{n,T}\bigcup\mathcal{N}^k_{n,T}$ is open and dense. 
\end{proof}
   Let $$\mathcal{R}=\bigcap^{\infty}_{k=1}\bigcap^{\infty}_{n=1}\bigcap^{\infty}_{T=1}\left(\mathcal{H}^k_{n,T}\cup \mathcal{N}^k_{n,T}\right).$$ 
   It is clear that $\mathcal{R}$ is a residual subset of $\mathcal{X}^1(M)$. For given $T>0$ and any $X\in \mathcal{R}$, if there exists a $C^1$ neighborhood $\mathcal{U}$ of $X\in \mathcal{R}$ such that any $Y\in \mathcal{U}$ has $k$ periodic points with period belonging to $(T,3T/2)$, then $X\notin \mathcal{N}^k_{n_0,T}$ for some $n_0$. Therefore, $X\in \mathcal{H}^k_{n_0,T}.$  
\end{proof}

   We prove Theorem \ref{Thm:non-star-infinite} based on the generic property of vector fields.  
\begin{proof}[Proof of Theorem \ref{Thm:non-star-infinite}]
   Take a dense $G_\delta$ set $\mathcal{R}\subset{\cal X}^1(M)$ as Lemma \ref{L8}. If any $X \in\mathcal{R}$ is not star, then for any $C^1$ neighborhood $\mathcal{U}$ of $X$, there is $Y\in\mathcal{U}$ which has a non-hyperbolic periodic point $x$ of $Y$. Let $T'$ be the period of $x$, then $\psi^Y_{T'}$ has a eigenvalue $\lambda$ whose module is $1$. For any $N\in\mathbb{N}^+$, we consider the following two cases.
	
\paragraph{$\lambda$ is real.} In this case, $\lambda=\pm 1$. We may assume that $\lambda=1$ ( The case $\lambda=-1$ can be proved similarly). After a $C^1$ perturbation, one can assume that $Y$ is locally linear in a small neighborhood of the periodic orbit. Therefore, there is an infinite subset $B\subseteq M$ such that $\varphi^Y_{T'}|B=Id$. We can find at least $e^N$ fixed points of $\varphi^Y_{T'}$ in a small cross section at $x$. By Lemma~\ref{L8}, $X$ has $e^{N}$ periodic orbits whose periods belong to $([T']-1,3([T'-1])/2)$. Consequently, one has 
$\dfrac{1}{2T'}\log\#P_{2T'}(X)\geq N/2$. 
	
\paragraph{$\lambda$ is not real.} In this case, $D\varphi^Y_{T'}$ is a rotation map on the sub-eigenspace $V$ w.r.t. the eigenvalue $\lambda$. One can also assume that $Y$ is locally linear in a small neighborhood of the periodic orbit after a $C^1$ perturbation, one can also assume that $\varphi^Y_{T'}|V$ is a rational rotation by perturbation. Thus there is $k\in\mathbb{N}^+$ such that $\varphi^Y_{kT'}|V=Id$. Therefore, one can find at least $e^{NT}$ fixed points of $\varphi^Y_{T}$, where $T=kT'$. Consequently, $\dfrac{1}{T}\log\#P_{T}(Y)\geq N$.

\bigskip
  
In any case, for any $C^1$ neighborhood $\mathcal{U}$ of $X$ and every $n\in\mathbb{N}^+$, there are $Y\in\mathcal{U}$ and $T_0=T_0(n)$ such that $Y$ has at least $e^{nT_0}$ periodic orbits whose periods is $T_0$. By Lemma \ref{L8}, $X$ has at least $e^{nT_0}$ periodic orbits whose periods belong to $(T_0,3T_0/2)$. By the arbitrariness of $n$, one has $$\limsup_{T\rightarrow \infty}\dfrac{1}{T}\log\#P_T(X)=+\infty.$$ 
\end{proof} 

\section{A shadowing lemma with time control}\label{Sec:shadow-measure}   
   For the linear Poincar\'e flow, one has the shadowing lemma of Liao for some quasi-hyperbolic orbit segments.

\begin{Definition}\label{Def:quasi-hyperbolic}
   Assume that $\Lambda\subset M\setminus{\rm Sing}(X)$ is an invariant (not necessarily compact) set having a dominated splitting $\cN_\Lambda=\cE\oplus\cF$ w.r.t. the linear Poincar\'e flow.  Given $\eta>0$ and $T_0 >0$, an orbit arc $\varphi^X_{[0 , T]}(x)\subset\Lambda$ with $T>T_0$ is \emph{$(\eta,T_0)$-quasi hyperbolic} (associated to $\Lambda$) if there is a time partition $0=t_0<t_1<t_2<t_3<\cdot\cdot\cdot<t_l=T$ with $t_{i+1}-t_i\leq T_0,~i=0,\cdots,l-1$ such that for $k = 0,1,\cdot\cdot\cdot,l-1$, one has
\begin{align*}
   \prod^{k-1}_{i=0}\parallel\psi^\ast_{t_{i+1}-t_i}|\mathcal {E}_{\varphi^X_{t_i}(x)}\parallel\leq e^{-\eta t_k},~\prod^{l-1}_{i=k}m(\psi^\ast_{t_{i+1}-t_i}|\mathcal {F}_{\varphi^X_{t_i}(x)})\geq {\rm e}^{\eta (T-t_k)}.
\end{align*}
\end{Definition}
   For obtaining a periodic orbit, we have the Liao's shadowing Lemma (see \cite[Theorem 5.5]{Lia85} and \cite[Theorem I]{Liao}) which means the recurrent quasi-hyperbolic orbits whose initial point and terminal point are far away from ${\rm Sing}(X)$ can be shadowed by periodic orbits.
\begin{Theorem}\label{Thm:Liao-shadowing}
   (Liao's shadowing Lemma) Suppose $\Lambda \subset M\backslash {\rm Sing}(X)$ is an invariant set with a dominated splitting $\cN_\Lambda=\cE\oplus\cF$. Given $\varepsilon_0>0$ and two constants $\eta >0, T_0\geq 1$, for every $\varepsilon > 0$, there is $\delta >0$ such that for any orbit segment $\varphi^X_{[0 , T]}(x)\subset\Lambda$ with the following properties:
\begin{itemize}
  \item $d(x,{\rm Sing}(X))\ge\varepsilon_0$ and $d(\varphi^X_T(x),{\rm Sing}(X))\ge\varepsilon_0$,

  \item $\varphi^X_{[0 , T]}(x)$ is $(\eta,T_0)$-quasi hyperbolic,

  \item $d(x,\varphi^X_T(x))<\delta$,
\end{itemize}
   Then there exist a $C^1$ increasing homeomorphism $\theta: [0 , T]\rightarrow \RR$ and a periodic point $p\in M$ whose period is $\theta(T)$ such that:
\begin{enumerate}
  \item $1-\varepsilon < \theta'(t)< 1+\varepsilon,~\forall~t\in[0,T];$ 
  \item $d(\varphi^X_t(x), \varphi^X_{\theta(t)}(p)) \leq \varepsilon\lvert X(\varphi^X_t(x))\rvert,~\forall~t\in[0,T].$
\end{enumerate}
\end{Theorem}

\begin{Remark-numbered}\label{Rem:normal}
In Theorem~\ref{Thm:Liao-shadowing}, in fact one can have that $$\varphi_{\theta(t)}^X(p)\in \exp_{\varphi_t(x)}\cN_{\varphi_t(x)}(2\varepsilon|X(\varphi_t(x))|)$$ for any $t\in[0,T]$, where $\cN_y(\chi)=\{v\in\cN_y:~|v|\le\chi\}$ for any regular point $y$ and any $\chi>0$.

\end{Remark-numbered}

In fact, one can get more information on the periodic orbit.

\begin{Proposition}\label{Pro:time-control}
   Under the setting of Theorem \ref{Thm:Liao-shadowing}, if $T=mT_0$ for some $m\in\mathbb{R}^+$, then there is $N=N(\eta,T_0)$ such that
    $$|\theta(\tau)-\tau|\leq N\cdot d(x,\varphi^X_T(x)),~~~\forall~\tau\in\NN\cap[0,T].$$
\end{Proposition}

For a normed vector space $A$ and $r>0$, denote by $A(r)=\{v\in A:~\|v\|<r\}$. Recall that $N_x(\chi)=\exp_x\cN_x(\chi)$.

\begin{Lemma}\label{L3}
   For the flow $\varphi^X_t$ generated by $X\in{\cal X}^1(M^d)$, there are two constants $C>0,~\delta>0$ such that if $y\in N_x(\delta| X(x)|)$, then there is a unique $t=t(y) \in [0,2]$ such that $\varphi_t(y)\in N_{\varphi_1(x)}(\delta)$ and $| t(y)-1|<C\cdot d(x,y)$.
\end{Lemma}

\begin{proof}
   We take $\varepsilon_0>0$ such that the exponential map $exp_x$ is a diffeomorphism on the ball $T_xM(\varepsilon_0)$. For any $x\in M$ and any $y$ close to $x$, one can lift the local orbit of $y$ to $T_xM$ in the following way: for any $v\in T_xM$, if $\lVert v\rVert<\varepsilon_0$, then one can define 
 $$\widetilde{\varphi}_t(v)= exp^{-1}_x\circ\varphi_t\circ exp_x(v).$$ 
   It is clear that $\widetilde{\varphi}_t$ is a local flow generated by a $C^1$ vector field $\widetilde{X}_x$ on $T_xM$, where 
 $${\widetilde X}_x(v)=D(exp_x^{-1})\circ X(exp_x(v)).$$
 Thus, we have that
 $$K=\sup_{x\in M,~v\in T_x M(\varepsilon_0)} \{|\widetilde{X}_x(v)|,~\Arrowvert D\widetilde{X}_x(v)\Arrowvert\}<+\infty.$$

Now for any regular point $x\in M$, one can identify $T_x M$ to $\RR^d$ by some isometrical transformation satisfying  $e_1=X(x)/\lvert X(x)\rvert$ for an orthonormal basis $e=\{e_1,\cdots,e_d\}$ of $\RR^d$. In this way, the  flow $\widetilde{\varphi}_t$ can be regard as the solution of differential equation: $dz/dt=\widetilde{X}_x(z)$.

Given $\varepsilon>0$, by reducing $\varepsilon_0$ if necessary, one can assume that the map $D\exp_x(v)$ is $\varepsilon$-close to the identity map for any $x$ and $v\in T_xM(\varepsilon_0)$.

\begin{Claim}
There is $\delta>0$ such that for any regular point $x$, one has that
$$N_x(\delta|X(x)|)\cap {\rm Sing}(X)=\emptyset.$$
\end{Claim}
 
\begin{proof}[Proof of the Claim]
 It suffices to consider the flow ${\widetilde \varphi}_t$ and the vector field ${\widetilde X}_x$. For $\delta<\varepsilon/K$, we have that for any $v\in N_x(\delta|X(x)|)$,
 $$|\widetilde{X}_x(v)|\ge |\widetilde{X}_x(0)|-\max_{\xi\in T_xM(\varepsilon_0)}\|D\widetilde{X}_x(\xi)\|.|v|\geq| X(x)|-K\delta|X(x)|\geq (1-\varepsilon)|X(x)|>0.$$
  Since the map $D\exp_x$ is $\varepsilon$-close to identity, one can conclude.
\end{proof}

\begin{Claim}
 By reducing $\delta$ if necessary, for any regular point $x$, any point $y\in N_x(\delta|X(x)|/2)$ and any $t\in[\delta/3,2\delta/3]$, there is a unique $s=s(t,y)\in[0,\delta]$ such that $\varphi_s(y)\in N_{\varphi_t(x)}(\delta)$.

\end{Claim}
\begin{proof}[Proof of the Claim]
As before, one can work in the local chart introduced above. By reducing $\delta$ if necessary such that for any $|v|\leq\delta|X(x)|$, one has 
 $$\sup_{t\in(-\delta,\delta)}\dfrac{| \widetilde{X}_x(v)|}{|\widetilde{X}_x(\widetilde{\varphi}_t(v))|}<1+\dfrac{\varepsilon}{K},~\sup_{t\in(-\delta,\delta)}\angle(\widetilde{X}_x(v),\widetilde{X}_x(\widetilde{\varphi}_t(v)))<\dfrac{\varepsilon}{K}.$$

For any $v\in N_x(\delta|X(x)|/3 )$, for the time $t_0$ satisfying
$$|\widetilde{\varphi}_{t_0}(v)|=\delta| {\widetilde X}_x({\widetilde\varphi}_{t_0}(0))|, ~{\textrm{and}}~|\widetilde{\varphi}_s(v)|\leq\delta|{\widetilde X}_x({\widetilde \varphi}_s(0))|,~~~\forall s\in [0,t_0),$$ 
one has 
$$|\widetilde{\varphi}_{t_0}(v)|=|v+\int^{t_0}_0\widetilde{X}_x(\widetilde{\varphi}_t(v))dt|\leq|v|+(1+\varepsilon)t_0|{\widetilde X}_x(0)|.$$
Thus,
$$\delta|{\widetilde X}_x(0)|\le \delta(1+\varepsilon/K)|{\widetilde X}_x(0)|\le \delta/3 |{\widetilde X}_x(0)|+(1+\varepsilon)t_0|{\widetilde X}_x(0)|.$$
Consequently, we have $t_0\geq\frac{2\delta}{3(1+\varepsilon)}\geq\frac{\delta}{2}$.  Similar estimate gives the fact that $t_0\le2\delta$ by reversing the inequalities.

%
%
%
  \end{proof}

The above claim allows one to define the non-linear dynamics along the flows: the \emph{sectional Poincar\'e map}  and the \emph{rescaled sectional Poincar\'e map}\footnote{These two maps have been theoretically studied in \cite{Yang} and \cite{CrY17}.}. For any regular point $x$ and $t$,  the sectional Poincar\'e map 
   $$\mathcal{P}_t:\mathcal{N}_x(\delta| X(x)|)\rightarrow\mathcal{N}_{\varphi_t(x)}(\delta)$$
 is the lift of the holonomy map induced by the local flow from $\exp_x(\mathcal{N}_x(\delta| X(x)|))$ to $\exp_{\varphi_t(x)}(\mathcal{N}_{\varphi_t(x)}(\delta))$.
   
The rescaled sectional Poincar\'e map
   $${\mathcal P}_t^*:\mathcal{N}_x(\delta)\rightarrow\mathcal{N}_{\varphi_t(x)}(\delta)$$
 is defined by 
 $${\cP}_t^*(v)=\frac{\cP_t(|X(x)|v)}{|X(\varphi_t(x))|},~~~\forall v\in\cN_x(\delta).$$

     In the local coordinate, by the choice of $e_1$, one can assume that $N_x(r)=\{\sum y_i e_i\}$, where $y=(y_2,\cdots,y_d)\in \RR^{d-1}$. By abuse of the notions, one also denotes $\sum y_i e_i$ by $y$. In this local coordinate, one can present $\cP_t^*$. Define the map $\tau:\mathcal{N}_x(\delta/2)\rightarrow\mathbb{R}^1$ such that 
$$\widetilde{\varphi}_{\tau(y)}\circ\widetilde{\varphi}_t(\|{\widetilde X}_x(x)\|y)/\|{\widetilde X}_x({\widetilde\varphi}_t(x))\|\in\cN_{\varphi_t(x)}(\delta)$$ 
for any $y\in \mathcal{N}_x(\delta/2)$. From the above facts, the map $\tau$ is an injective.
In the local chart, the rescaled sectional Poincar\'e map $\cP_t^*$ can be written by
$$\cP_{t}^*(y)=\widetilde{\varphi}_{\tau(y)}\circ\widetilde{\varphi}_t(\|{\widetilde X}_x(x)\|y)/\|{\widetilde X}_x({\widetilde\varphi}_t(x))\|.$$

We are going to estimate $\tau(y)$. For $t\in[\delta/3,2\delta/3]$, we consider the function
    $$H(x,t,y,\tau)=\langle\widetilde{\varphi}_\tau\circ\widetilde{\varphi}_t(|{\widetilde X}_x(x)|y)/|{\widetilde X}_x({\widetilde\varphi}_t(x))|,\dfrac{\widetilde{X}_x(\widetilde{\varphi}_t(x))}{|\widetilde{X}_x(\widetilde{\varphi}_t(x))|}\rangle,$$ where $\langle\cdot,\cdot\rangle$ denotes the inner product in the local Euclidean coordinate. From the definition of map $\tau$, one has 
  \begin{enumerate}
 \item $H(x,t,y_0,\tau(y_0))=0$ from the definition of map $\tau$.
  	
 \item $\partial H/\partial y$ and $\partial H/\partial \tau$ are equi-continuous w.r.t. $x$ and $t$.
  	
  \item $\partial H/\partial\tau\arrowvert_{y=0,\tau=0}=\langle\dfrac{\widetilde{X}_x(\widetilde{\varphi}_t(x))}{| \widetilde{X}_x(\widetilde{\varphi}_t(x))|},\dfrac{\widetilde{X}_x(\widetilde{\varphi}_t(x))}{|\widetilde{X}_x(\widetilde{\varphi}_t(x))|}\rangle=1$.
  \end{enumerate}
By the Implicit Function Theorem, one has $\dfrac{\partial \tau}{\partial y}=-\dfrac{\partial H/\partial y}{\partial H/\partial \tau}$. Since $\varphi_t$ is $C^1$ and $\varphi_t(y)\in T_xM(\varepsilon_0)$, $\partial H/\partial\tau$ is uniformly bounded away from zero and  $\partial H/\partial y$ is uniformly bounded. Therefore, $\partial \tau/\partial y$ is uniformly bounded w.r.t. $y$. This means that $| s(t,y)-t|=| \tau(y)|\leq C_0\cdot d(x,y)$, where $C_0$ is a constant decided by $\partial \tau/\partial y$.      
  
 By adding the time consecutively, one can get the constant $C$. 
  
\end{proof}
   Fixed $\eta>0,~T_0\geq 1$, the following Lemma~\ref{L4} will show that the distance between $(\eta,T_0)$-quasi hyperbolic orbit and its shadowing periodic orbit can be controlled by the distance between the starting point and terminal point of this quasi hyperbolic orbit. The idea about proof of Lemma \ref{L4} can refer to \cite[Page 269, Corollary 6.4.17]{Katok1}. 
  
\begin{Lemma}\label{L4}
  Under the assumption of Proposition \ref{Pro:time-control}, taking $\alpha=e^{-\eta/2}$, there is a constant $C>0$ such that for the $(\eta,T_0)$-quasi hyperbolic orbit $\varphi_{[0,mT_0]}(x)$ and the shadowing orbit $\varphi_{[0,\theta(mT_0)]}(p)$ in Theorem \ref{Thm:Liao-shadowing}, one has $$d(\varphi_{iT_0}(x),\varphi_{\theta(iT_0)}(p))\leq C\cdot\alpha^{min\{i,m-i\}}\cdot d(x,\varphi_{mT_0}(x)),~\forall~i\in [0,n].$$
\end{Lemma}

\begin{proof}
By Remark~\ref{Rem:normal}, one can assume that $\varphi_{\theta(t)}^X(p)\in \exp_{\varphi_t(x)}\cN_{\varphi_t(x)}(2\varepsilon|X(\varphi_t(x))|)$ for any $t\in[0,T]$. Thus, one can lift the dynamics in the normal fibers and consider a sequence of rescaled sectional Poincar\'e maps $\{\cP_{T_0,\varphi_{iT_0}(x)}^*:\cN_{\varphi_{iT_0}(x)}\to \cN_{\varphi_{(i+1)T_0}(x)}\}$. Since $\varphi_{[0,mT_0]}(x)$ is a $(\eta,T_0)$-quasi hyperbolic orbit, by the Definition \ref{Def:quasi-hyperbolic}, for the dominated splitting $\cN_\Lambda=\cE\oplus\cF$ w.r.t. the linear Poincar\'e flow, one has $$\prod^{k-1}_{i=0}\parallel\psi^\ast_{T_0}|\mathcal {E}_{\varphi_{(i-1)T_0}(x)}\parallel\leq e^{-k\eta },~\prod^{m-1}_{i=k}m(\psi^\ast_{T_0}|\mathcal {F}_{\varphi_{iT_0}(x)})\geq {\rm e}^{(m-k)\eta},~{\rm for  }~k=0,\cdots,m.$$ 
By Liao's shadowing Lemma \ref{Thm:Liao-shadowing}, one has $d(\varphi_t(x), \varphi_{\theta(t)}(p)) \leq \varepsilon\lvert X(\varphi_t(x))\rvert,~\forall~t\in[0,T]$. For the maps $\{\cP_{T_0,\varphi_{iT_0}(x)}^*\}$, the periodic orbit of $p$ also admits a dominated splitting $\cN_x=\cE_p\oplus \cF_p$ with respect to $\{D\cP_{T_0,\varphi_{iT_0}(x)}^*\}$. Thus, they have plaques in the normal fibers. The distance in the $\cE$-plaques are denoted by $d_\cE$ and the the distance in the $\cF$-plaques are denoted by $d_\cF$. One can split the distance into $\cE$-distance and $\cF$-distance.  There is a constant $C\ge1$ such that $d_\cF(p,x)\le C d(p,x)$, $d_\cE(p,x)\le C d(p,x)$ and $d(p,x)\le d_\cE(p,x)+d_\cF(p,x)$. Therefore, 
\begin{eqnarray*}
d(\varphi_{iT_0}(x),\varphi_{\theta(iT_0)}(p))&\leq& d_{\cE}(\varphi_{iT_0}(x),\varphi_{\theta(iT_0)}(p))+d_{\cF}(\varphi_{iT_0}(x),\varphi_{\theta(iT_0)}(p))\\
&\leq& \alpha^id_\cE(x,p)+\alpha^{m-i}d_\cF(\varphi_{mT_0}(x),\varphi_{\theta(mT_0)}(p))\\
&\leq&\alpha^id(x,p)+\alpha^{m-i}d(\varphi_{mT_0}(x),\varphi_{\theta(mT_0)}(p))\\
&\leq& \alpha^{\min\{i,m-i\}}(d(x,p)+d(\varphi_{mT_0}(x),\varphi_{\theta(mT_0)}(p)))
\end{eqnarray*}	
By adapting a generalized shadowing lemma of S. Gan \cite[Theorem 1.1]{Gan02}, by enlarging $C$ if necessary, we have that $$d(x,p)\leq C d(x,\varphi_T(x))~\text{and}~d(\varphi_T(x), \varphi_{\theta(T)}(p)) \leq Cd(x,\varphi_T(x))$$ Therefore, $$d(\varphi_{iT_0}(x),\varphi_{\theta(iT_0)}(p))\leq C^2\cdot\alpha^{min\{i,m-i\}}\cdot d(x,\varphi_{mT_0}(x)),~\forall~i\in [0,m].$$
\end{proof}

\begin{proof}[Proof of Proposition \ref{Pro:time-control}]
  By Theorem \ref{Thm:Liao-shadowing}, given $\varepsilon_0>0$, $\eta >0$ and $T_0\geq 1$, for every $\varepsilon > 0$, there is $\tilde{\delta} >0$ such that for any $(\eta,T_0)$-quasi hyperbolic orbit segment $\varphi^X_{[0 , T]}(x)\subset\Lambda$ with $d(x,{\rm Sing}(X))\ge\varepsilon_0$, $d(\varphi_T(x),{\rm Sing}(X))\ge\varepsilon_0$ and $d(x,\varphi_T(x))<\tilde{\delta}$, there is a periodic point $p\in M$ and a $C^1$ strictly increasing function $\theta$ such that $\varphi_{\theta(T)}(p)=p$ and $d(\varphi_t(x),\varphi_{\theta(t)}(p))\leq \varepsilon\lvert X(\varphi_t(x))\rvert$ for any $t\in [0,T]$. Consider a time partition $0=t_0<t_1<t_2<t_3<\cdot\cdot\cdot<t_m=T$ with $t_{i+1}-t_i=T_0$. Taking $\alpha={\rm e}^{-\eta/2}\in(0,1)$, by Lemma \ref{L4}, there is $C_1>0$ such that $$d(\varphi_{\theta(t_i)}(p),\varphi_{t_i}(x))\leq C_1\cdot\alpha^{\min\{i,m-i\}}\cdot d(x,\varphi_T(x)),~i=0,1,\dots,m.$$ On the fact that $d(\varphi_t(x),\varphi_{\theta(t)}(p))\leq \varepsilon\lvert X(\varphi_t(x))\rvert$ for every $t\in [0,T]$, by Lemma~\ref{L3}, there is $C_2=C_2(T_0)>0$ such that for $\varphi_{\theta(t_i)}(p)\in N_{\varphi_{t_i}(x)}(\delta),~i=0,1,\dots,m$, one has $$|\theta(t_{i+1})-\theta(t_i)-(t_{i+1}-t_i)|\leq C_2\cdot d(\varphi_{t_i}(x),\varphi_{\theta(t_i)}(p)),~\forall~i=0,1,\dots,m.$$Let $N=N(\eta,T_0)=\dfrac{2 C_2\cdot C_1}{1-\alpha}$, one has    
\begin{eqnarray*}
  |\theta(T)-T|&\leq&\sum^{m-1}_{i=0}|\theta(t_{i+1})-\theta(t_i)-(t_{i+1}-t_i)|~\leq~C_2\cdot\sum^{m-1}_{i=0}d(\varphi_{t_i}(x),\varphi_{\theta(t_i)}(p))\\
  &\leq&C_2\cdot C_1\cdot d(\varphi_T(x),x)\sum^{m-1}_{i=0}\alpha^{\min\{i,m-i\}}
  \leq N \cdot d(\varphi_T(x),x).
\end{eqnarray*}
From the above discussion, for any $\tau\in\NN\cap[0,T]$ $$|\theta(\tau)-\tau|\leq\sum^{m-1}_{i=0}|\theta(t_{i+1})-\theta(t_i)-(t_{i+1}-t_i)|\leq N\cdot d(x,\varphi^X_T(x)).$$
\end{proof}

\section{Periodic Orbits of Vector Fields by Shadowing}\label{Sec:periodic orbits}

\subsection{Pesin block of vector fields} 
   For a regular hyperbolic ergodic measure $\mu$ and its hyperbolic Oseledec splitting $\mathcal{N}=\mathcal{E}^s\oplus\mathcal{F}^u$, by the definition of $\psi_t^*$, for $\mu$-a.e. $x$, one has
\begin{equation} \label{E1}
\begin{split} 
  \lambda^-(\mu)=\lim_{t\rightarrow\infty}\dfrac{1}{t}\log\parallel\psi^*_t|\mathcal{E}^s_x\parallel<0,~~~\lambda^+(\mu)=\lim_{t\rightarrow\infty}\dfrac{1}{t}\log m(\psi^*_t|\mathcal{F}^u_x)>0
\end{split} 
\end{equation}
    
\begin{Lemma}\label{L6}
   If the hyperbolic Oseledec splitting of a regular hyperbolic ergodic measure $\mu$ is a dominated splitting, then for any $\varepsilon >0,$ there exists $T(\varepsilon) \in \mathbb{R}$ such that for $\mu$-$a.e.~x\in M$ and every $T\geq T(\varepsilon),$ one has 
  $$\lim_{k\to\infty}\dfrac{1}{k\cdot T}\sum^{k-1}_{i=0}\log\parallel\psi^*_T|\mathcal{E}^s_{\varphi_{iT}(x)}\parallel {\rm exists \ and \ is \ contained \ in \ } [\lambda^-(\mu),\lambda^-(\mu)+\varepsilon),$$ 
  $$\lim_{k\to\infty}\dfrac{1}{k\cdot T}\sum^{k-1}_{i=0}\log\parallel\psi^*_{-T}|\mathcal{F}^u_{\varphi_{-iT}(x)}\parallel {\rm exists \ and \ is \ contained \ in \ } (-\lambda^+(\mu)-\varepsilon,-\lambda^+(\mu)].$$   
\end{Lemma}

\begin{proof}
   Let $R$ be the support of $\mu$ and $\widetilde{R}$ be the transgression of $R$. By Lemma \ref{Lem:dominated-splitting-closure}, $\widetilde R$ admits a dominated splitting ${\cal N}_{\widetilde R}SM={\mathcal E}^s\oplus {\mathcal F}^u$ w.r.t. the extended linear Poincar\'e flow. By Lemma~\ref{ms}, one has $$\lambda^-(\mu)=\lambda^-(\widetilde\mu)=\lim_{t\rightarrow +\infty}\frac{1}{t}\int \log\parallel\widetilde{\psi}_t|\mathcal{E}^s\parallel {\rm d}\widetilde\mu,$$ where $\widetilde\mu$ is the transgression of $\mu$. Therefore, for any $\varepsilon >0$, there is $T(\varepsilon)>0$ large enough such that for every $T\geq T(\varepsilon)$, one has $|\dfrac{1}{T}\int \log\parallel\widetilde{\psi}_T|\mathcal{E}^s_x\parallel {\rm d}\widetilde\mu-\lambda^-(\mu)|<{\varepsilon}$. By Lemma \ref{lemma: continuous of elp}, since every $\frac{X(x)}{\parallel X(x)\parallel}\in\widetilde R$ is a unit vector, for the fixed $T\geq T(\varepsilon)$, $\widetilde{\psi}_T(\cdot)$ is continuous on $NSM \triangleq\{(v_1,v_2):v_1\in S_xM,~v_2\in T_xM,~v_1\perp v_2\}$. Thus, by the Birkhoff ergodic theorem, one has $$\lim_{k\rightarrow\infty}\dfrac{1}{k\cdot T}\sum^{k-1}_{i=0}\log\parallel\psi^*_T|\mathcal{E}^s_{\varphi_{iT}(x)}\parallel=\dfrac{1}{T}\int\log\parallel\widetilde\psi_T|\mathcal{E}^s_x\parallel d\widetilde\mu <\lambda^-(\mu)+\varepsilon.$$ Since the norms are sub-multiplicative, one has $$\lim_{k\rightarrow\infty}\dfrac{1}{k\cdot T}\sum^{k-1}_{i=0}\log\parallel\psi^*_T|\mathcal{E}^s_{\varphi_{iT}(x)}\parallel\geq\lim_{t\rightarrow\infty}\dfrac{1}{t}\log\parallel\psi^*_t|\mathcal{E}^s_x\parallel=\lambda^-(\mu).$$ The conclusion for sub-bundle $\mathcal{F}^u$ can be obtained similarly.     
\end{proof}

\begin{Definition} \label{pesin block}
  Let $\mu$ be a regular hyperbolic ergodic measure of $X\in{\cal X}^1(M)$, ${\cal N}_\Lambda={\mathcal E}^s\oplus {\mathcal F}^u$ be the hyperbolic Oseledec splitting, where $\Lambda$ is a Borel set with $\mu$-total measure. Given $\lambda\in (0,1)$, $L>0$ and $k\in\mathbb{R}^+$, the \emph{Pesin block} $\Lambda^L_\lambda(k)$ is defined as:
\begin{eqnarray*}
  \Lambda_\lambda^L(k) = \{x\in \Lambda: &\prod\limits^{n-1}_{i=0}&\parallel   \psi^*_{L}|{\mathcal{E}^s_{\varphi_{iL}(x)}}\parallel \leq k \lambda^n,~\forall~n\geq 1,\\
  &\prod\limits^{n-1}_{i=0}&\parallel\psi^*_{-L}|{\mathcal{F}^u_{\varphi_{-iL}(x)}}\parallel \leq k \lambda^n,~\forall~n\geq 1,~d(x,{\rm Sing}(X)\cap \Lambda)\geq\frac{1}{k}~\}.
\end{eqnarray*} 
\end{Definition}
        
\begin{Proposition}\label{P1}
  If the hyperbolic Oseledec splitting of a regular hyperbolic measure $\mu$ is a dominated splitting, then the \emph{Pesin block} $\Lambda^L_\lambda(k)$ is a compact set such that
  $$\mu(\Lambda^L_\lambda(k))\rightarrow 1 {\rm \ as \ } k\to +\infty.$$ where $\lambda={\rm e}^{-\eta}, 0<\eta<\min\{-\lambda^-(\mu),\lambda^+(\mu)\}$, $L\geq T(\min\{-\lambda^-(\mu),\lambda^+(\mu)\}-\eta)$ as in Lemma \ref{L6}. 
\end{Proposition}

\begin{proof}
  By Definition \ref{pesin block}, $\Lambda^L_\lambda(k)$ is a compact set. By the hyperbolicity of the regular measure and the choice of $\lambda$, for $\mu$-a.e. $x\in M$, by Lemma \ref{L6}, there is $C(x)$ satisfying$$\prod^{n-1}_{i=0}\parallel   \psi^*_{L}|{\mathcal{E}^s_{\varphi_{iL}(x)}}\parallel \leq C(x)\cdot \lambda^n,~\prod^{n-1}_{i=0}\parallel\psi^*_{-L}|{\mathcal{F}^u_{\varphi_{-iL}(x)}}\parallel \leq C(x)\cdot \lambda^n,~\forall ~n\geq 1.$$  Let $\Gamma^L_\lambda(k) = \{x\in \Lambda: \prod^{n-1}_{i=0}\parallel \psi^*_L|{\mathcal{E}^s_{\varphi_{iL}(x)}}\parallel \leq k \lambda^n, \prod^{n-1}_{i=0}\parallel\psi^*_{-L}|{\mathcal{F}^u_{\varphi_{-iL}(x)}}\parallel \leq k \lambda^n, \forall~n\geq 1\}$. Therefore, $$\mu(\bigcup_{k>0}\Gamma^L_\lambda(k))=1.$$ For two real numbers $0<k_1<k_2$, one has $\Gamma^L_\lambda(k_1)\subset\Gamma^L_\lambda(k_2)$. Consequently, $$\mu(\Gamma^L_\lambda(k))\rightarrow 1~~~ {\rm as }~~~k\rightarrow\infty.$$ According to the facts that $\Lambda^L_\lambda(k)\subset \Gamma^L_\lambda(k)$ and $\mu({\rm Sing}(X))=0$, for any $\varepsilon>0$, there is $K=K(\varepsilon)\in\mathbb{N}$ such that $| \mu(\Gamma^L_\lambda(k))-\mu(\Lambda^L_\lambda(k))| <\varepsilon,~\forall~k\geq K$. Then $$\mu(\Lambda^L_\lambda(k))\rightarrow 1~{\rm as }~k\rightarrow\infty.$$   	
\end{proof}

\subsection{Constructing many periodic orbits: proof of Theorem \ref{Thm:measure-exponential}  }
   We have the following version of Poincar\'e Recurrence Theorem for flows. It can be deduced by the case of diffeomorphisms. Hence, the proof is omitted.
\begin{Proposition}\label{Pro:poincare-recurrence}
   Let $\mu$ be an $\varphi_t$-invariant measure. For any fixed time $t_0$ and any set $B$ with positive $\mu$-measure, there is a set $R\subset B$ with $\mu(R)=\mu(B)$ and a sequence of integers $0<n_1< n_2< n_3<\dots<n_i<\cdots$ such that for every $x\in R,$  
\begin{enumerate}
  \item $\varphi_{n_i\cdot t_0}(x)\in R,$ for all $i\in \mathbb{N};$ 	
  \item $d(x,\varphi_{n_i\cdot t_0}(x))\rightarrow 0,$ as $i\rightarrow\infty;$
\end{enumerate}
\end{Proposition}

In fact, one can have the following stronger recurrent property

\begin{Proposition}\label{Pro:strong-recurrence}
Assume that $f$ is a homeomorphism on a compact metric space $M$. Let $\mu$ be an ergodic invariant measure of $f$. If $\Lambda$ is a set with positive measure of $\mu$, then given $\delta>0$ and $\varepsilon>0$, we have that
$$\lim_{n\to\infty}\mu(\Lambda_n)=\mu(\Lambda),$$
where
$$\Lambda_n=\{x\in\Lambda:~\exists~ m\in[n,(1+\varepsilon)n],~f^m(x)\in\Lambda,~d(f^m(x),x)<\delta\}.$$
\end{Proposition}

\begin{proof}Given $\delta>0$ and $\varepsilon>0$, take a finite measurable partition ${\cal P}=\{P_i\}^{\ell}_{i=1}$ such that $${\rm diam} (P_i)\leq \delta,~P_i\subset \Lambda{\rm \ or \ } P_i\cap \Lambda=\emptyset,~\text{for $ i=1,2,\cdots,\ell$}.$$
Consider the set $$\Lambda_n({\cal P})\triangleq\{x\in\Lambda: \exists~i\in[1,\ell] {\rm~and}~m\in[n,(1+\varepsilon)n], {\rm s.t.}~f^m(x)\in \Lambda,~x,f^m(x)\in P_i\in{\mathcal P}\}.$$ Fix $P_i\subset\Lambda$, define $${\cal P}^i_{n,\varepsilon}=\{x\in P_i: \sum^{n-1}_{j=0}\chi_{P_i}(f^j(x))\leq n\mu(P_i)(1+\dfrac{\varepsilon}{3}),\sum^{[n(1+\varepsilon)]}_{j=0}\chi_{P_i}(f^j(x))\geq n\mu(P_i)(1+\dfrac{2\varepsilon}{3})\},$$ where $\chi_{P_i}$ is the characteristic function of set $P_i$. Therefore, ${\cal P}^i_{n,\varepsilon}\subset P_i\bigcap\Lambda_n({\cal P})$. By the Birkhoff ergodic theorem, $\mu(P_i\setminus{\cal P}^i_{n,\varepsilon})\rightarrow 0$ as $n\rightarrow\infty$. This implies the proposition.
\end{proof}

Now we are ready to prove Theorem \ref{Thm:measure-exponential}.
\begin{proof}[Proof of Theorem \ref{Thm:measure-exponential}]The proof will follow some steps.
\paragraph{Choose a Pesin block.}
   Since the hyperbolic Oseledec splitting $\mathcal{N}=\mathcal{E}^s\oplus\mathcal{F}^u$ of the ergodic hyperbolic measure $\mu$ is a dominated splitting, one can take $\lambda\in(0,1)$ which is bigger than and close to ${\rm e}^{-\min\{-\lambda^-(\mu),~\lambda^+(\mu)\}}$ and $L$ as Lemma \ref{L6} to define Pesin block (Definition \ref {pesin block}) $\Lambda^L_\lambda(k):$ 
\begin{eqnarray*}
 \Lambda_\lambda^L(k) = \{x\in \Lambda: &\prod\limits^{n-1}_{i=0}&\parallel   \psi^*_{L}|{\mathcal{E}^s_{\varphi_{iL}(x)}}\parallel \leq k \lambda^n,~\forall~n\geq 1,\\
 &\prod\limits^{n-1}_{i=0}&\parallel\psi^*_{-L}|{\mathcal{F}^u_{\varphi_{-iL}(x)}}\parallel \leq k \lambda^n,~\forall~n\geq 1,~d(x,{\rm Sing}(X)\cap \Lambda)\geq\frac{1}{k}~\}.
\end{eqnarray*}
   By Proposition \ref {P1}, taking $k>0$ large enough such that $\mu(\Lambda_\lambda^L(k))>0$. One can fix $\lambda_0\in (\lambda,1)$, then there is $j=j(k)\in \mathbb{N}^+$ such that for any $x\in\Lambda^L_\lambda(k)$, $$\prod\limits^{n-1}_{i=0}\parallel   \psi^*_{j(k)L}|{\mathcal{E}^s_{\varphi_{ij(k)L}(x)}}\parallel \leq\lambda_0^n,~~\prod\limits^{n-1}_{i=0}\parallel\psi^*_{-j(k)L}|{\mathcal{F}^u_{\varphi_{-ij(k)L}(x)}}\parallel \leq \lambda_0^n,~\forall~ n\geq 1.$$ Then, for the set $\Lambda^{L_0}_{\lambda_0}(k)$ defined by 
\begin{eqnarray*}
 \Lambda_{\lambda_0}^{L_0}(k) \triangleq \{x\in \Lambda: &\prod\limits^{n-1}_{i=0}&\parallel   \psi^*_{j(k)L}|{\mathcal{E}^s_{\varphi_{ij(k)L}(x)}}\parallel \leq  \lambda_0^n,~\forall~n\geq 1,\\
 &\prod\limits^{n-1}_{i=0}&\parallel\psi^*_{-j(k)L}|{\mathcal{F}^u_{\varphi_{-ij(k)L}(x)}}\parallel \leq  \lambda_0^n,~\forall~n\geq 1,~d(x,{\rm Sing}(X)\cap \Lambda)\geq\frac{1}{k}~\},
\end{eqnarray*} 
   where $L_0=jL$, we have $\mu(\Lambda_{\lambda_0}^{L_0}(k))\ge\mu(\Lambda_\lambda^L(k))>0$. Hereafter we fix this $k$. By Proposition \ref{Pro:poincare-recurrence}, taking $B=\Lambda_{\lambda_0}^{L_0}(k)$, one has that for $\mu$-a.e. $x\in \Lambda_{\lambda_0}^{L_0}(k)$, the forward orbit of $x$ will return to $\Lambda_{\lambda_0}^{L_0}(k)$ and will be arbitrarily close to $x$. Let $\eta_0=-\log\lambda_0$. If $\varphi_{nL_0}(x)\in\Lambda_{\lambda_0}^{L_0}(k)$ for some $n\in{\mathbb N}$, then $\varphi_{[0,nL_0]}(x)$ is $(\eta_0, L_0)$-quasi hyperbolic orbit arc.

\paragraph{The shadowing constants.}
    Let $C=\max\{1,\max_{x\in M}\lvert X(x)\rvert\}$. Given $\varepsilon_0=1/k$, $\eta=\eta_0$, $T_0=L_0$ and $\varepsilon>0$, for $\varepsilon_1=\varepsilon/3C$, by Theorem \ref{Thm:Liao-shadowing}, there is $\delta=\delta(\varepsilon)$ much smaller than $\varepsilon$ such that for any $x,\varphi_{nL_0}(x)\in\Lambda_{\lambda_0}^{L_0}(k)$, if $d(x,\varphi_{nL_0}(x))<\delta$, then there is a point $p\in M$ and $C^1$ strictly increasing function $\theta: [0,nL_0]\rightarrow \mathbb{R}$ such that $p=\varphi_{\theta(n L_0)}(p)$ and $d(\varphi_t(x),\varphi_{\theta(t)}(p))\leq\varepsilon_1\lvert X(\varphi_t(x))\rvert<\varepsilon/3$ for all $t\in[0,nL_0]$. Moreover, by Proposition \ref{Pro:time-control}, one has $|\theta(t)-t|\le N\delta$ for any integer $t\in[0,nL_0]$, where $N$ is a constant which is independent of $x$ and $n$. One can also assume that $N\delta$ is much smaller than $\varepsilon$.

\paragraph{A seperation set $K_n$.} For $\varepsilon>0$ and $n\in {\mathbb N}$, we claim that there is a finite set $K_n=K_n(k,\varepsilon)\subset \Lambda^{L_0}_{\lambda_0}(k)$ with the following properties:
\begin{enumerate}
  \item\label{I.1} For points $x,y\in K_n$, there is an integer  $t\in[0,nL_0]$ such that $d(\varphi_t(x),\varphi_t(y))>\varepsilon$;

  \item\label{I.2} For any $x\in K_n$, there is an integer $m=m(n)$ with $n<m\leq(1+\varepsilon)n$ such that $\varphi_{mL_0}(x)\in \Lambda^{L_0}_{\lambda_0}(k)$ and $d(x,\varphi_{mL_0}(x))<\delta(\varepsilon)$;

  \item\label{I.3} $\liminf\limits_{\varepsilon\rightarrow 0}\liminf\limits_{n\to\infty}\dfrac{1}{nL_0}\log \#K_{n}\geq {\rm h}_{\mu}(\varphi_1).$
\end{enumerate}    
\paragraph{The construction of $K_n$.} Now, we give the precise construction of $K_n$.   We consider the following sets:
$$\Lambda_{\lambda_0}^{L_0}(k,n)=\{x\in\Lambda_{\lambda_0}^{L_0}(k):~\exists m\in[n,(1+\varepsilon)n],~\varphi_{mL_0}(x)\in\Lambda,~d(x,\varphi_{mL_0}(x))<\delta(\varepsilon)\}.$$
By the construction, we have that $\mu(\Lambda_{\lambda_0}^{L_0}(k))>0$. Apply Proposition~\ref{Pro:strong-recurrence} by taking $f=\varphi_{L_0}$, one has
$$\lim_{n\to\infty}\mu(\Lambda_{\lambda_0}^{L_0}(k,n))=\mu(\Lambda_{\lambda_0}^{L_0}(k)).$$ 

We take a maximal choice of $K_n=K_n(k,\varepsilon)\subset \Lambda_{\lambda_0}^{L_0}(k,n)$ such that Item~\ref{I.1} is satisfied. By the definition of $\Lambda_{\lambda_0}^{L_0}(k,n)$, Item~\ref{I.2} is satisfied.

For Item~\ref{I.3}, let us recall the definition of the metric entropy by Katok \cite[Theorem 1.1]{Katok}, see also Section~\ref{Sec:reduction}. By the maximality of $K_n$, one has 
$$\#K_{n}\geq  N_{\varphi_1}(nL_0,\varepsilon,1-\mu(\Lambda^{L_0}_{\lambda_0}(k,n))).$$ 
 Thus, $$\liminf_{\varepsilon\rightarrow 0}\liminf_{n\to\infty}\dfrac{1}{nL_0}\log \#K_{n}\geq \liminf_{\varepsilon\rightarrow 0}\liminf_{n\to\infty}\dfrac{1}{nL_0}\log N_{\varphi_1}(nL_0,\varepsilon,1-\mu(\Lambda^{L_0}_{\lambda_0}(k,n)))\geq {\rm h}_{\mu}(\varphi_1).$$ The construction of $K_n$ is hence complete. 
\paragraph{Estimate the growth rate of the periodic orbits.}   
Now we can complete the poof of Theorem \ref{Thm:measure-exponential}. For every $x\in K_n$, there is $m_x$ with $n\le m_x\leq n(1+\varepsilon)$ such that $\varphi_{m_xL_0}(x)\in \Lambda_{\lambda_0}^{L_0}(k)$. By Theorem \ref{Thm:Liao-shadowing}, there is a strictly increasing map $\theta_x:[0,m_xL_0]\rightarrow\mathbb{R}$ and a periodic point $p=p_x$ of period $\theta(m_xL_0)$ such that 
$$d(\varphi_t(x),\varphi_{\theta_x(t)}(p))<\varepsilon_1\lvert X(\varphi_t(x))\rvert<\varepsilon/3,~~~\forall t\in[0,m_xL_0].$$ 
 By Proposition~\ref{Pro:time-control}, one has that 
 $$| \theta_x(\tau)-\tau|\leq N\cdot d(x,\varphi_{mL_0}(x))\leq N\delta,~~~\forall \tau\in\NN\cap [0,mL_0].$$

For two different points $x,y\in K_n$, by the construction of $K_n$, there is $j\in\NN\cap [0,nL_0]$ such that $d(\varphi_j(x),\varphi_j(y))>\varepsilon$. Thus,
\begin{eqnarray*}
  d(\varphi_{\theta_x(j)}(p_x),\varphi_{\theta_y(j)}(p_y))&\geq& d(\varphi_{j}(x),\varphi_{j}(y))-d(\varphi_{j}(x),\varphi_{\theta_x(j)}(p_x))-d(\varphi_{j}(y),\varphi_{\theta_y(j)}(p_y))\\
  &>& \varepsilon-\varepsilon/3-\varepsilon/3=\varepsilon/3.
\end{eqnarray*}    
In fact, we have the following disjoint property:
\begin{Claim}
$\varphi_{(-\varepsilon/32C,~\varepsilon/32C)}(p_x)\cap\varphi_{(-\varepsilon/32C,~\varepsilon/32C)}(p_y)=\emptyset$, where $C=\sup\limits_{z\in M}\{\lVert X(z)\rVert\}<\infty$.	
\end{Claim}   
\begin{proof}[Proof of the Claim] 
By Proposition \ref{Pro:time-control}, taking $\delta\in(0,\varepsilon/64CN)$, one has $|\theta_x(j)-j|\le N\delta$ and $|\theta_y(j)-j|\le N\delta$. Therefore, 
$$|\theta_x(j)-\theta_y(j)|\le |\theta_x(j)-j|+|\theta_y(j)-j|\le 2N\delta<\varepsilon/32C.$$ 
Thus, one has that for any $t\in(-t_0, t_0)$, 
$$d(\varphi_{\theta_x(j)+t}(p_x),\varphi_{\theta_x(j)}(p_x))<\varepsilon/16 ~~~{\rm and } ~~~d(\varphi_{\theta_y{(j)+t}}(p_y),\varphi_{\theta_y(j)}(p_y))<\varepsilon/16.$$
 Consequently, for any $t,s\in(-t_0,t_0)$,
\begin{eqnarray*}
   d(\varphi_{\theta_x(j)+t}(p_x),\varphi_{\theta_y(j)+s}(p_y))&\geq& d(\varphi_{\theta_x(j)}(p_x),\varphi_{\theta_y(j)}(p_y))-d(\varphi_{\theta_x(j)+t}(p_x),\varphi_{\theta(j)}(p_x))\\&-&d(\varphi_{\theta_y(j)+s}(p_y),\varphi_{\theta_y(j)}(p_y))> \varepsilon/3-\varepsilon/4>0.
\end{eqnarray*}    
This implies the claim.
\end{proof} 
From the claim, in ${\rm Orb}(p_x)$, any orbit segment $\varphi_{[0,1]}(z)$ contains at most $32C/\varepsilon$ points in the set $\{p_x\}_{x\in K_n}$. Consequently, we have that
$$\sum_{[x]\in P_T(X),~nL_0(1-\varepsilon)-N\delta\le \pi(x)\le nL_0(1+\varepsilon)+N\delta} \pi(x)\ge \varepsilon/32C.\#K_n.$$

Thus,
$$\#P_{nL_0(1+\varepsilon)+N\delta}(X)\geq\varepsilon/32C.\#K_n..$$
 Therefore, 
\begin{eqnarray*}
\limsup_{T\to\infty}\frac{1}{T}\log\# P_T(X)&\ge&\limsup_{n\to\infty}\frac{1}{nL_0(1+\varepsilon)+N\delta}\log\# P_{nL_0(1+\varepsilon)+N\delta}(X)\\
&=&\lim_{n\to\infty}\frac{nL_0}{nL_0(1+\varepsilon)+N\delta}\limsup_{n\to\infty}\frac{1}{nL_0}\log\# P_{nL_0(1+\varepsilon)+N\delta}(X)\\
&\ge&\frac{1}{1+\varepsilon}\liminf_{n\to\infty}\frac{1}{nL_0}\log(\varepsilon/32C.\#K_n)\\
&\ge&\frac{1}{1+\varepsilon}\liminf_{n\to\infty}\frac{1}{nL_0}\log\#K_n
\end{eqnarray*} 
   Thus, by letting $\varepsilon\to 0$, one has $$\limsup_{T\rightarrow\infty}\frac{1}{T}\log \#P_T(X)~~\geq~~h_\mu(X).$$   
\end{proof} 
Theorem \ref{Thm:measure-exponential} is now proved, hence the proof of the main theorem is complete.

\paragraph{Acknowledgements.}
We are grateful to S. Crosivier and X. Wen
for useful communications.

\vskip 5pt

\noindent Wanglou Wu

\noindent School of Mathematical Sciences

\noindent Soochow University, Suzhou, 215006, P.R. China

\noindent wuwanlou@163.com

\vskip 5pt

\noindent Dawei Yang

\noindent School of Mathematical Sciences

\noindent Soochow University, Suzhou, 215006, P.R. China

\noindent yangdw1981@gmail.com, yangdw@suda.edu.cn

\vskip 5pt

\noindent Yong Zhang

\noindent School of Mathematical Sciences

\noindent Soochow University, Suzhou, 215006, P.R. China

\noindent yongzhang@suda.edu.cn

\end{document}